\documentclass[11pt]{amsart}
\usepackage{lmodern}
\usepackage{amsmath, amsthm, amssymb}
\usepackage{amsfonts}
\usepackage[normalem]{ulem}
\usepackage{hyperref}

\usepackage[utf8]{inputenc}

\usepackage[margin=3.5cm]{geometry}

\usepackage{verbatim} 
\usepackage{longtable}
\usepackage{faktor}
\usepackage{mathtools}

\usepackage{tikz}
\usetikzlibrary{decorations.pathmorphing}
\tikzset{snake it/.style={decorate, decoration=snake}}

\usepackage{caption}
\usepackage{tikz-cd}
\usetikzlibrary{arrows}

\usepackage{diagbox}
\usepackage{booktabs}

\usepackage{pdflscape}

\DeclareSymbolFont{bbold}{U}{bbold}{m}{n}
\DeclareSymbolFontAlphabet{\mathbbold}{bbold}

\newtheoremstyle{mytheorem}
{3pt}
{3pt}
{\itshape}
{}
{\bf}
{.}
{.5em}
{}

\theoremstyle{mytheorem}

\newtheorem{thm}{Theorem}[section]
\newtheorem{cor}[thm]{Corollary}

\theoremstyle{definition}

\newtheorem{example}[thm]{Example}

\theoremstyle{remark}

\newcommand{\CC}{{\mathbb C}}

\newcommand{\QQ}{\mathbb{Q}}

\DeclareFontFamily{OT1}{rsfs}{}
\DeclareFontShape{OT1}{rsfs}{n}{it}{<-> rsfs10}{}
\DeclareMathAlphabet{\curly}{OT1}{rsfs}{n}{it}

\newcommand{\ZZ}{\mathbb{Z}}

\newcommand{\rk}{\operatorname{rk}}

\usepackage{setspace}

\begin{document}

\title{Configuration Spaces of \texorpdfstring{$\mathbb{C} \setminus k$}{C-k}}
	
\date{\today}

\author{Christoph Schiessl}
\address{ETH Z\"urich, Department of Mathematics}
\email{christoph.schiessl@math.ethz.ch}

\thanks{ The author was supported by the grant ERC-2012-AdG-320368-MCSK.
	I want to thank Emanuele Delucchi, Emmanuel Kowalski, Johannes Schmitt, Junliang Shen for very helpful discussions and especially Rahul Pandharipande for his invaluable support.
	This is part of the author's PhD thesis.
}

\begin{abstract}
In this note, we collect mostly known formulas and methods to compute the standard and virtual Poincaré polynomials of the configuration spaces of the plane $\CC \setminus k$ with $k$ deleted points and compare the answers.
\end{abstract}

\baselineskip=14.5pt

\maketitle

\tableofcontents
For any complex quasi-projective algebraic variety $X$, the \emph{virtual Poincaré polynomial} $S(X) \in \mathbb{Z}[x]$ is defined \cite{newton}, \cite{totaroserre} by the properties 
\begin{itemize}
\item $S(X) = \sum \rk H^i(X) \, x^i$ for smooth, projective $X$,
\item $S(X) = S(X \setminus C) +S(C)$ for a closed subvariety  $C \subset X$,
\item $S(X \times Y ) = S(X) S(Y)$.
\end{itemize}
In contrast, we write \[ P(X) = \sum \rk H^i(X) x^i \] for the standard Poincaré polynomial.

For any space $X$, the ordered configuration space 
\[ F_n(X ) = \{ x_1, \dots, x_n \in X^n | \, x_i \neq x_j \} \] is the space of $n$ distinct points in $X$. The symmetric group $S_n$ acts on $F_n(X)$ by permuting the points and the quotient \[ C_n(X) = F_n(X) / S_n \] is the unordered configuration space. Computing their cohomology is a classical, hard problem. As the configuration space is the complement of diagonals, determining the virtual Poincaré polynomials is simpler and was done for example by Getzler \cite{getzler} \cite{getzler2}.

We look at the cohomology of ordered and unordered configuration spaces of $\mathbb{C} \setminus k$. We compute their normal and virtual Poincaré Polynomials by existing methods and see that Stirling and pyramidal numbers show up. The calculation for $C_n(\CC \setminus k)$ seems not to be in the literature in this form.

\section{Pyramidal Numbers}
The \emph{$k$-dimensional pyramidal numbers} are integers $P_{k,i}$ for $i \ge -1$, $k \ge -1$. They satisify the recursions \begin{align*} P_{-1, i} = \begin{cases} 1 & i=0 \\ 0 & \text{otherwise} \end{cases} & & P_{k+1, i} = \sum_{j=0}^{i} P_{k, j}. \end{align*} An equivalent recursion would be \begin{align*} P_{k,0}= 1 & & P_{k+1, i+1} = P_{k,i+1}+P_{k+1,i}. \end{align*} Some examples are \begin{align*} P_{0,i} = 1 & & P_{1,i} = i+1 & & P_{2,i} = \frac{(i+1)(i+2)}{2}.\end{align*} 
	\begin{table}[h]
		\caption*{Some pyramidal numbers $P_{k,i}$:}
		\centering
	\begin{tabular}{c|ccccc}
	\diagbox{k}{i}        & 0 & 1 & 2 & 3 & 4 \\
		\midrule
			       -1 & 1 & 0 & 0 & 0 & 0 \\
				0 & 1 & 1 & 1 & 1 & 1 \\
				1 & 1 & 2 & 3 & 4 & 5 \\
				2 & 1 & 3 & 6 & 10 & 15 \\
				3 & 1 & 4 & 10 & 20 & 35
	\end{tabular}
\end{table}
	
The recursion allows us to compute the generating function \[ \sum P_{k, i} x^i = ( 1+x+x^2+x^3+x^4+ \dots)^{k+1}= \frac{1}{(1-x)^{k+1}}. \] By standard manipulation of generating series for $k \ge0$: \[ \frac{1}{(1-x)^{k+1}} = \frac{1}{k!}\frac{d^k}{dx^k} \frac{1}{1-x} = \frac{1}{k!} \sum_{i \ge 0} (i+k) \dots (i+2)(i+1)x^i = \sum_{i \ge 0} \binom{i+k}{i} x^i \] 
The result \[ P_{k,i} = \binom{i+k}{i} \] also holds for $k=0$ and can be proved directly using the recursion \[ P_{k+1, i+1} = \binom{i+k+2}{i+1} = \binom{i+k+1}{i+1} + \binom{i+k+1}{i} = P_{k,i+1}+P_{k+1,i}.\]
The definition could be extended by setting \[ P_{k,i} =0 \text{ for } i<0.\] In this way, all recursions stay valid for $i <0$.

\section{Poincaré Polynomials of \texorpdfstring{$C_n(\mathbb{C} \setminus k)$}{}}

Let $M$ be a connected manifold. Napolitano \cite[Theorem 2]{napolitano} proved the following relation between the cohomology of unordered configuration spaces of $M \setminus 1$ and $M \setminus 2$:
\[ H^j(C_n(M  \setminus 2), \ZZ) = \bigoplus_{t=0}^n H^{j-t}(C_{n-t}(M \setminus 1, \ZZ)). \] We use the conventions \begin{align*} H^0(C_0(M \setminus 1), \ZZ) = \ZZ & & H^j(C_0(M \setminus 1), \ZZ)= 0 \text{ if } j > 0. \end{align*} In general, this relation does not hold between the cohomology of the configuration spaces of $M \setminus 1$ and $M$ as the proof works by pushing in points from the missing point.
\begin{thm}
	We have \[ \rk H^i(C_n(\CC \setminus k), \ZZ) = \begin{cases} P_{k-1, i} & i=n \\ P_{k-1,i}+ P_{k-1,i-1 } & 0 \le i < n \\ 0 & \text{otherwise} \end{cases} \]
or 
	\[ \sum_{ n \ge 0} P(C_n(\mathbb{C} \setminus k)) y^n = \frac{1+xy^2}{(1-y)(1-xy)^k}.\]
\end{thm}

\begin{proof}
	Write \[Q_k(x,y) = \sum_{n, i \ge 0} \rk H^i(C_n(\mathbb{C} -k), \ZZ) \, x^i y^n .\] Then applying Napolitano's recursion to $M=S^2 \setminus k+1$ we get 
	\begin{equation*} Q_{k+1}(x,y) = \frac{Q_k(x,y)}{1-xy}. \end{equation*}

Arnold's computation 
	\begin{align*} H^0(C_n(\CC), \QQ) = \QQ & & H^1(C_n(\CC), \QQ) = \QQ \text{ if } n \ge 2 & & H^i(C_n(\CC), \QQ) = 0 \text{ if } i \ge 2 \end{align*} 
	in \cite{Arn} provides initial values for $k=0$:
	\[ Q_0(x,y) = 1 + y + (1+x)y^2 +(1+x)y^3 + \dots = \frac{1+xy^2}{1-y} \]
	Hence we have shown 
	\[ Q_k(x,y) = \frac{1+xy^2}{(1-y)(1-xy)^k}.\]
	Expansion now proves the theorem.
\end{proof}

This theorem can also be deduced from \cite[Prop. 3.5]{knudsen}. As $C_1(\CC \setminus k ) = \CC \setminus k$, the reality check for $n=1$ works:
	\[ \rk H^j(C_1(\CC \setminus k), \ZZ)  = \begin{cases} 1 & \text{for }  j=0 \\ k & \text{for } j=1 \\ 0 & \text{otherwise} \end{cases}.\]
	We can conclude that $\rk H^j(C_n(\CC \setminus k), \ZZ)$ stabilizes (seen as a function of $n$) for $n>j$.
\begin{cor} \label{gen}
In the limit we get
	\[ \rk H^j(C_\infty(\CC \setminus k), \ZZ) = P_{k-1, j}+P_{k-1, j-1}\] or as a generating series
	\begin{equation*} P(C_\infty(\mathbb{C} \setminus k))= \frac{1+x}{(1-x)^{k}} . \end{equation*}
\end{cor}
Taking stability for granted, this can be deduced by the stable version of Napolitano's recursion:
\[H^j(C_\infty(\CC \setminus k+1), \ZZ) = \bigoplus_{t=0}^j \rk H^t(C_\infty(\CC \setminus k), \ZZ).\]
	Vershinin \cite[Cor. 11.1]{versh} showed that
	\[ H^*(C_\infty(\mathbb{C} \setminus k) \simeq H^*(\Omega^2 S^3) \otimes \left ( H^*(\Omega S^2) \right )^k\] extending the May-Segal formula \cite{segal1}, \cite[Th. 8.11]{versh} \[ H^*(C_\infty(\mathbb{C}) \simeq H^*(\Omega^2 S^3). \] Combining the results of Arnold and the cohomology of the loop spaces of a sphere \[ H^i(\Omega S^2) = \mathbb{Z} \] for $i \ge 0$ \cite[Example 1.5]{hatcher}), this gives back corollary (\ref{gen}).

\section{Poincaré Polynomials of \texorpdfstring{$F_n(\CC \setminus k)$}{}}

Arnold's calculation of $H^*(F_n(\mathbb{C}), \ZZ)$ can be extended to $H^*(F_n(\CC \setminus k), \ZZ)$ via the fiber bundles
\begin{equation*} \label{fiber} F_n(\CC \setminus k) \mapsto F_{n-1}(\CC \setminus k) \end{equation*} with fiber $ \CC \setminus (k+n-1)$.
	\begin{thm} \cite[Thm. 7.1]{vershcol}
We have \[P(F(\mathbb{C} \setminus k, n)) = (1+kx)(1+(k+1)x) \cdots (1+(n+k-1)x).\]
\end{thm}

\section{Virtual Poincaré Polynomials of \texorpdfstring{$F_n(\mathbb{C} \setminus k)$}{} }

We have \[S(\mathbb{C} \setminus k) = S(\mathbb{CP}^1 \setminus k+1) = x^2+1-(k+1) = x^2-k.\] Using the same fiber bundles or \cite[Theorem, page 2]{getzler} we get
\begin{thm} The virtual Poincaré polynomials of $F_n(\CC \setminus l)$ is given by
	\[S(F_n(\mathbb{C} \setminus k)) = (x^2-k)(x^2-k-1)\cdots (x^2-k-n+1).\]
\end{thm}

\section{Virtual Poincaré Polynomials of \texorpdfstring{$C_n(\CC \setminus k)$}{}}

As $S( \mathbb{C} \setminus k) = (x^2-k)$,
the calculations of Getzler \cite[Cor. 5.7]{getzler} allow us to conclude 
	\[ \sum_{n \ge 0} S(C_n(\mathbb{C} \setminus k)) \, y^n = \frac{(1-y^2x^2)(1-y)^k}{(1-yx^2)(1-y^2)^k}, \]
	which simplifies to
\begin{thm} \cite{getzler}
	The virtual Poincaré polynomials of $C_n(\mathbb{C} \setminus k)$ are given by the following generating series:
		\[ \sum_{n \ge 0} S(C_n(\mathbb{C} \setminus k)) \, y^n = \frac{(1-y^2x^2)}{(1-yx^2)(1+y)^k}\]
\end{thm}

	\section{Comparision}

We observe that under the variable transformation \[ x \to -1/x^2, y \to yx^2\] 
the respective generating series
\begin{align*}  \sum_{n \ge 0} P(C_n(\CC \setminus k)) \, y^n  & & \sum_{n \ge 0}  P(F_n(\CC \setminus k)) \, y^n \end{align*}
transform into 
\begin{align*}  \sum_{n \ge 0} S(C_n(\CC \setminus k)) \, y^n  & & \sum_{n \ge 0} P(F_n(\CC \setminus k)) \, y^n. \end{align*}
This means, in this case the classical and virtual Poincaré polynomials are in some sense dual to each other.
\begin{example}
We look 3-pointed configuration spaces of $\CC \setminus 2$:

	\begin{align*}  P(C_3(\CC \setminus 2)) = 4x^3+5x^2+3x+1 & & P(F_3(\CC \setminus 2))= 24x^3+26x^2+9x+1 \end{align*}

		\begin{align*}  S(C_3(\CC \setminus 2)= x^6-3x^4+5x^2-4  & & S(F_3(\CC \setminus 2)) = x^4-9x^4+26x^2-24 \end{align*}
\end{example}


\bibliographystyle{alpha}
\bibliography{thesis}
\end{document}